\documentclass[10pt,a4paper,oneside]{article}

\usepackage{amssymb}
\usepackage{amsmath}
\usepackage{amsthm}

\newtheorem{theorem}{Theorem}[section]
\newtheorem{corollary}{Corollary}[section]
\theoremstyle{remark} \newtheorem{remark}{Remark}[section]

\newcommand{\MN}{\mathcal{M}_{n}}
\newcommand{\TMN}{\mathop{M}_{n}}
\newcommand{\T}{\mathop{\textrm{T}}\nolimits}
\newcommand{\dint}{\mathop{\mathrm{d}}\nolimits}
\newcommand{\Dom}{\mathop{\rm{Dom}}\nolimits}
\newcommand{\di}{\displaystyle}
\newcommand{\GL}{\mathop{\rm{GL}}\nolimits}
\newcommand{\Mat}{\mathop{\textrm{Mat}}\nolimits}
\newcommand{\Lin}{\mathop{\textrm{Lin}}\nolimits}
\newcommand{\Tr}{\mathop{\rm{Tr}}\nolimits}
\newcommand{\Scal}{\mathop{\rm{Scal}}\nolimits}
\newcommand{\Ric}{\mathop{\rm{Ric}}\nolimits}

\newcommand{\un}[1]{\underline{#1}}
\newcommand{\ce}{\mathop{\textrm{e}}\nolimits}
\newcommand{\F}{\mathop{\textrm{F}}\nolimits}

\newcommand{\EN}{\mathop{\textrm{T}}\nolimits}
\newcommand{\RE}{\mathop{\textrm{R}}\nolimits}
\newcommand{\KM}{\mathop{\textrm{KM}}\nolimits}
\newcommand{\LA}{\mathop{\textrm{La}}\nolimits}
\newcommand{\SI}{\mathop{\textrm{CO}}\nolimits}
\newcommand{\LMR}{\mathop{\textrm{LMR}}\nolimits}

\newcommand{\Aut}{\mathop{\textrm{Aut}}\nolimits}
\newcommand{\SO}{\mathop{\textrm{SO}}\nolimits}
\newcommand{\SL}{\mathop{\textrm{SL}}\nolimits}

\begin{document}

\title{
  On the geometry of generalized Gaussian distributions
  \thanks{keywords: Gaussian distribution, differential geometry;
          MSC: 94A17, 53B21}}
\author{Attila Andai\thanks{andaia@math.bme.hu}\\
  RIKEN, BSI, Amari Research Unit \\
  2--1, Hirosawa, Wako, Saitama 351-0198, Japan.}
\date{May 11, 2007}

\maketitle

\begin{abstract}
In this paper we consider the space of those probability distributions which maximize
  the $q$-R\'enyi entropy.
These distributions have the same parameter space for every $q$, and in the $q=1$ case
  these are the normal distributions.
Some methods to endow this parameter space with Riemannian metric is presented: the second
  derivative of the $q$-R\'enyi entropy, Tsallis-entropy and the relative entropy give rise
  to a Riemannian metric, the Fisher-information matrix is a natural Riemannian metric, and
  there are some geometrically motivated metrics which were studied by Siegel, Calvo and Oller,
   Lovri\'c, Min-Oo and Ruh.
These metrics are different therefore our differential geometrical calculations based
  on a unified metric, which covers all the above mentioned metrics among others.
We also compute the geometrical properties of this metric, the equation of the geodesic line
  with some special solutions, the Riemann and Ricci curvature tensors and scalar curvature.
Using the correspondence between the volume of the geodesic ball and the scalar curvature
  we show how the parameter $q$ modulates the statistical distinguishability of close points.
We show that some frequently used metric in quantum information geometry can be easily recovered
  from classical metrics.
\end{abstract}

\section{Introduction}

In theoretical statistics and in applications the distance functions between probability
  distributions play an important role.
The construction of a proper distance function has been considered by several authors.
But even the same statistical model with different mathematical frameworks can lead to
  different distance functions.
To narrow the family of potential distance functions we consider those which are natural from
  differential geometrical point of view.

Historically the pioneering work of Mahalanobis \cite{Mah} was generalized by Rao \cite{Rao},
  who first suggested the idea of considering the Fisher information \cite{Fis} as a Riemannian
  metric on the space of probability distributions.
Cencov \cite{Cen} was the first to study monotone metrics on statistical manifolds.
He proved that, up to a normalization, there exists a unique monotone metric, the Fisher
  information.
Amari \cite{Ama1} and Amari and Nagaoka \cite{AmaNag} provide modern account of the general
  differential geometry that arises from the Fisher information metric.
The Fisher metric was studied further by Akin \cite{Aki}, James \cite{Jam},
  Burbea \cite{Bur}, Mitchell \cite{Mit}, Atkinson and Mitchell \cite{AtkMit},
  Skovgaard \cite{Sko}, Oller \cite{Oll}, Oller and Cuadrasa \cite{OllCua},
  Oller and Corcuera \cite{OllCor} among other researchers.
The combination of differential geometrical and statistical studies helped to find the
  statistical interpretation of geometrical quantities.
For example the geodesic distance between probability distributions,
  which is usually known as Rao distance is a natural distance function between probability distributions;
  the statistical meaning of the so-called e-curvature was first clarified by Efron \cite{Efr};
  the normalized volume measure of the manifold is called Jeffreys' prior \cite{Jef} within the field of
  Bayesian statistics.

In this paper we consider the space of those probability distributions which maximize the
  $q$-R\'enyi entropy.
These distributions have the same parameter space for every $q$, and in the $q=1$ case these
  are the normal distributions.
The first results about the geometrical properties of these spaces are due to
  Amari \cite{Ama1,Ama2}.
He considered the Fisher information metric on these manifolds and computed some geometrical
  invariants.
Some methods to endow the parameter space with Riemannian metric is presented: the second
  derivative of the $q$-R\'enyi entropy \cite{Ren}, Tsallis-entropy \cite{Tsa1} and the
  relative entropy give rise to a Riemannian metric, the Fisher-information matrix is a
  natural Riemannian metric, and there are some geometrically motivated metrics which were
  studied by Siegel \cite{Sie}, Calvo and Oller \cite{CalOll} and Lovri\'c,
  Min-Oo and Ruh \cite{LovMinRuh}.
These metrics are different therefore our differential geometrical calculations based
  on a unified metric, which covers all the above mentioned metrics among others.
We also compute the geometrical properties of this metric, the equation of the geodesic line
  with some special solutions, the Riemann and Ricci curvature tensors and scalar curvature.
Using the correspondence between the volume of the geodesic ball and the scalar curvature
  we show how the parameter $q$ modulates the statistical distinguishability of close points.
We show that some frequently used metric in quantum information geometry can be easily recovered
  from classical metrics.

\section{$q$-R\'enyi entropy maximizing distributions}

The normal distributions can be introduced as a result of the maximum entropy principle.
Consider the family of density functions which are continuous and supported on the real line
  with given expectation value $\mu\in\mathbb{R}$ and variance $\sigma^{2}\in\mathbb{R}$.
Introducing the Lagrange multipliers $a,b,c$ we have the following functional on the family of
  probability distributions
\begin{align*}
S(p)=&-\int p(x)\log p(x)\ \dint x-a\left(\int p(x)\ \dint x-1 \right)\\
 &-b\left(\int p(x)x\ \dint x-\mu\right)-c\left(\int p(x)(x-\mu)^{2}\ \dint x-\sigma^2 \right).
\end{align*}
The variation of the functional is
\begin{equation*}
\delta S=\int \left(-\log p(x)-1-a-bx-c(x-\mu)^{2} \right)\delta p(x)\ \dint x.
\end{equation*}
The functional has extremal point at $p$ if its variation is zero.
One can show that the entropy functional has local maximum at the point
\begin{equation*}
p(x)=\exp\left(-a-bx-c(x-\mu)^2\right)
\end{equation*}
  for appropriate parameters $a,b,c\in\mathbb{R}$.

The family of one dimensional normal distributions $S_{1}$ can be parameterized by the
  expectation value $u\in\mathbb{R}$ and the parameter $d\in\mathbb{R}^{+}$ as
\begin{equation*}
f(d,u,x)=\frac{\sqrt{d}}{\sqrt{2\pi}}\ce^{-\frac{1}{2}d(x-u)^{2}}.
\end{equation*}
This means that $S_{1}$ can be identified with a $2$ dimensional space
  $\Xi_{1}=\mathbb{R}^{+}\times\mathbb{R}$.
The statistical properties of the distributions lead us to define Riemannian metric on the
  space $\Xi_{1}$.

In general, the family of $n$ dimensional normal distributions $S_{n}$ can be parameterized by
  the expectation vector $\un{u}\in\mathbb{R}^{n}$ and the inverse of the covariance matrix $D$.
Let us denote the set of real symmetric strictly positive $n\times n$ matrices by $\MN$.
Then we can identify the sets $S_{n}$ and $\Xi_{n}=\MN\times\mathbb{R}^{n}$ using the following
  one-to-one map
\begin{equation*}
\Xi_{n}\to S_{n}\qquad (D,\un{u})\mapsto f(D,\un{u},\cdot),
\end{equation*}
where
\begin{equation*}
f(D,\un{u},\cdot):\mathbb{R}^{n}\to\mathbb{R}\qquad
  \un{x}\mapsto\frac{\sqrt{\det D}}{\sqrt{(2\pi)^{n}}}
  \exp\left(-\frac{1}{2}\langle\un{x}-\un{u},D(\un{x}-\un{u})\rangle\right)
\end{equation*}
Normal distributions with zero expectation will said to be to special normal distributions.
The parameter space of the $n$ dimensional special normal distribution is $\Xi^{(s)}_{n}=\MN$.

One can generalize the above mentioned procedure to extend the notion of Gaussian distributions
  using the $q$-R\'enyi entropy \cite{Ren}.
Let us fix a parameter $q\in\mathbb{R}^{+}\setminus\left\lbrace 1\right\rbrace$ and consider a
  density function $p$.
The $q$-R\'enyi entropy of the distribution $p$ is
\begin{equation*}
S_{q}(p)=\frac{1}{1-q}\log\int_{\mathbb{R}} p(x)^{q}\ \dint x
\end{equation*}
  if the integral exists.

The $q$-R\'enyi entropy maximizing distribution is the following.
For a given $n\in\mathbb{N}\setminus\left\lbrace 0\right\rbrace$ the parameter space is
  $\Xi_{n}=\MN\times\mathbb{R}^{n}$.
For a parameter $(D,\un{u})\in\Xi$ define the set
\begin{equation*}
\Dom(p,D,\un{u})=\left\lbrace
\begin{array}{cll}
\mathbb{R}^{n},& \mbox{if} &\di p\in\left\rbrack \frac{n}{n+2},1\right\lbrack;\\
\di  \left\lbrace \un{x}\in\mathbb{R}^{n}\ \vert\
  1+\frac{1-p}{2p-n(1-p)}\langle\un{x}-\un{u},D(\un{x}-\un{u})\rangle\geq 0 \right\rbrace, &
    \mbox{if} & p>1;
\end{array}\right.
\end{equation*}
  and define the density function as
\begin{equation*}
f_{p}(D,\un{u},\un{x})=\left\lbrace
\begin{array}{cll}
  \di A_{n,p}\sqrt{\det{D}}
  \left(1+\frac{1-p}{2p-n(1-p)}\langle\un{x}-\un{u},D(\un{x}-\un{u})\rangle\right)^{\frac{1}{p-1}},
  &\mbox{if} &\di \un{x}\in\Dom(p,D,\un{u});\\
  0,&\mbox{if} &\di \un{x}\notin\Dom(p,D,\un{u}).
  \end{array}\right.
\end{equation*}

The normalization constant of the generalized $p$-Gaussian distribution is
\begin{equation*}
\di A_{n,p}=\left\lbrace\begin{array}{lll}
\di\left(\frac{1-p}{2p-n(1-p)}\right)^{\frac{n}{2}}
  \frac{\Gamma\left(\frac{1}{1-p}\right)}{\pi^{\frac{n}{2}}
  \Gamma\left(\frac{1}{1-p}-\frac{n}{2}\right)},
  &\di\mbox{if}&\di p\in\left\rbrack \frac{n}{n+2},1\right\lbrack;\\
\di\left(\frac{p-1}{2p-n(1-p)}\right)^{\frac{n}{2}}
  \frac{\Gamma\left(\frac{p}{p-1}+\frac{n}{2}\right)}{\pi^{\frac{n}{2}}
  \Gamma\left(\frac{p}{p-1}\right)},
  &\di\mbox{if}&\di p>1.
\end{array}\right.
\end{equation*}

For a given parameters $n\in\mathbb{N}\setminus\left\lbrace 0\right\rbrace$ and $p>\frac{n}{n+2}$ we call
\begin{equation*}
M_{p}=\left\lbrace f_{p}(D,\un{u},\cdot)\ \vert\ (D,\un{u})\in\Xi_{n} \right\rbrace
\end{equation*}
  the family of $p$-generalized Gaussian distributions.
This forms a manifold parameterized by $(D,\un{u})$.
This is an $\alpha$-family of probability distributions, where $\alpha=2p-1$ and is $\alpha$-flat
  (see Amari and Nagaoka \cite{AmaNag}).
The present paper studies the geometrical structures of $M_{p}$.

If we consider the limit $q\to 1$ then the $q$-R\'enyi entropy tends to the entropy.
From this point on we will allow the $p=1$ case, and we will consider it as the usual Gaussian
  distribution, and in the $p=1$ case we sometimes omit the index $p$.
The set
\begin{equation*}
\mathcal{N}=\left\lbrace (n,p)\in\mathbb{N}\times\mathbb{R}\ \vert\
 n>0,\ p>\frac{n}{n+2}\right\rbrace
\end{equation*}
  can be considered as the label set of the $p$-Gaussian distributions, and for every pair
  $(n,p)\in\mathcal{N}$ the parameter space of the $n$-dimensional $p$-Gaussian distributions
  is $\Xi_{n}=\MN\times\mathbb{R}^{n}$ and the parameter space of the special Gaussian
  distributions is $\Xi_{n}^{(s)}=\MN$.

We present a Theorem which shows the maximum $q$-R\'enyi entropy property of the $p$-Gaussian
  distributions in the $q=p$ case.
The maximum R\'enyi entropy problem was solved by Moriguti in the scalar case \cite{Mor}.
The distribution function was remarked by Zografos \cite{Zog} in the multivariate case, but not
  connected to the R\'enyi entropy.
The problem was solved first by Kapur \cite{Kap} in the multivariate case,
  Johnson and Vignat also solved the problem in the multivariate case \cite{JohVig} using the
  result of Lutwak, Yang and Zhang \cite{LutYanZha}.
Costa, Hero and Vignat \cite{CosHerVig} established properties of multivariate distributions
  maximizing R\'enyi-entropy, under a covariance constraint.

\begin{theorem}
For any probability density $g:\mathbb{R}^{n}\to\mathbb{R}^{+}$ with fixed covariance matrix
  $K$, expectation $\un{u}\in\mathbb{R}^{n}$ and parameter $q>\frac{n}{n+2}$,
\begin{equation*}
S_{q}(g)\leq S_{q}(f_{q}(K^{-1},\un{u},\cdot)),
\end{equation*}
  with equality if and only if $g=f_{q}(K^{-1},\un{u},\cdot)$ almost everywhere.
\end{theorem}

Important to note, that the $p$-Gaussian distributions maximize not only the $q$-R\'enyi entropy,
  but the Tsallis entropy too, defined by equation (\ref{eq:Tsallis entropy}) and
  minimize $\alpha$-relative entropy (defined in the next Section)
  between the uniform distribution and an arbitrary one.

We call the family of probability distributions $f_{p}(D,\un{u},\cdot)$ ($(n,p)\in\mathcal{N}$,
  $(D,\un{u})\in\Xi_{n}$) extended Gaussian distributions.

\section{Riemannian metrics on the space of extended Gaussian distributions}

The parameter spaces $\Xi_{n}$ and $\Xi^{(s)}_{n}$ have a natural manifold structure.
Let us denote the space of real symmetric $n\times n$ matrices by $\TMN$.
Then at the point $(D,\un{u})\in\Xi_{n}$ the tangent space $\T_{(D,\un{u})}\Xi_{n}$ can be
  identified by $\T_{n}=\TMN\times\mathbb{R}^{n}$, since one can consider the tangent vector
  $(X,\un{x})$ as a derivation defined for any smooth function $h:\Xi\to\mathbb{R}$ as
\begin{equation}
\label{eq:deriv_def}
\frac{\partial h(D,\un{u})}{\partial (X,\un{x})}=
  \left.\frac{\dint}{\dint t}h(D+tX,\un{u}+t\un{x})\right|_{t=0}.
\end{equation}
In this setting a map
\begin{equation*}
g:\Xi_{n}\times\T_{n}\times\T_{n}\to\mathbb{C}\qquad
 ((D,\un{u}),(X,\un{x}),(Y,\un{y}))\mapsto g_{D,\un{u}}((X,\un{x}),(Y,\un{y}))
\end{equation*}
  will be called a Riemannian metric if the following conditions hold.
For all $(D,\un{u})\in\Xi_{n}$ the map
\begin{equation*}
g_{D,\un{u}}:\T_{n}\times\T_{n}\to\mathbb{C}\qquad
  ((X,\un{x}),(Y,\un{y}))\mapsto g_{D,\un{u}}((X,\un{x}),(Y,\un{y}))
\end{equation*}
  is a scalar product and for all $(X,\un{x})\in\T_{n}$ the map
\begin{equation*}
g((X,\un{x}),(X,\un{x})):\Xi_{n}\to\mathbb{C}\qquad
  (D,\un{u})\mapsto g_{D,\un{u}}((X,\un{x}),(X,\un{x}))
\end{equation*}
  is smooth.

Now we present some ideas how the space $\Xi_{n}$ can be endowed with Riemannian metric.
For example the ($q$-R\'enyi) entropy can generate a Riemannian metric: because the following
  Theorem shows that $q$-R\'enyi entropy is a convex functional, so its second derivative is
  a strictly positive symmetric linear map, and therefore it can define a Riemannian metric.

\begin{theorem}
For every pair $(n,p)\in\mathcal{N}$ and $(D,\un{x})\in\Xi_{n}$ the
  $q$-R\'enyi entropy ($q\in\mathbb{R}^{+}$) of the distribution $f_{p}(D,\un{u},\cdot)$ is
\begin{align}
&\di\mbox{if}\ p,q>1:\nonumber\\
&\di\label{eq:SR1_p(f_q)} S_{q}(f_{p}(D,\un{u},\cdot))=\frac{n}{2}\log\frac{\pi(2p-n(1-p))}{p-1}
 +\frac{1}{1-q}\log\left\lbrack
 \frac{\Gamma\left(\frac{p}{p-1}+\frac{n}{2} \right)^{q}\Gamma\left(\frac{q}{p-1}+1 \right)}
 {\Gamma\left(\frac{p}{p-1}\right)^{q}\Gamma\left(\frac{q}{p-1}+1+\frac{n}{2}\right)}
  \right\rbrack-\frac{1}{2}\log\det D\\
&\di\mbox{if}\ p<1,\ q>\frac{n(1-p)}{2}:\nonumber\\
&\di\label{eq:SR2_p(f_q)} S_{q}(f_{p}(D,\un{u},\cdot))=\frac{n}{2}\log\frac{\pi(2p-n(1-p))}{1-p}
 +\frac{1}{1-q}\log\left\lbrack
 \frac{\Gamma\left(\frac{1}{1-p}\right)^{q}\Gamma\left(\frac{q}{1-p}-\frac{n}{2} \right)}
 {\Gamma\left(\frac{1}{1-p}-\frac{n}{2}\right)^{q}\Gamma\left(\frac{q}{1-p}\right)}
  \right\rbrack-\frac{1}{2}\log\det D.
\end{align}
\end{theorem}
\begin{proof}
First we compute the integral
\begin{equation}
\label{eq:intf_p^q}
I=\int_{\Dom(p,D,\un{u})}f_{p}(D,\un{u},\un{x})^{q}\ \dint\un{x}.
\end{equation}
Choose our new coordinate system in $\mathbb{R}^{n}$ parallel to the eigenvectors of $D$.
In this coordinate system $D$ is diagonal, with entries $(\lambda_{i})_{i=1,\dots,n}$.
If $p>1$ then with the variables $a=\frac{p-1}{2p-n(p-1)}$ and
  $y_{i}=\sqrt{a\lambda_{i}}(x_{i}-u_{i})$ the integral is
\begin{equation*}
I=\frac{A_{n,p}^{q}(\det D)^{\frac{q-1}{2}}}{a^{\frac{n}{2}}}
  \int_{B_{n}(1)}\left(1-\sum_{k=1}^{n}y_{k}^{2} \right)^{\frac{q}{p-1}} \ \dint\un{y}.
\end{equation*}
In spherical coordinates this equation is
\begin{equation*}
I=\frac{A_{n,p}^{q}(\det D)^{\frac{q-1}{2}}}{a^{\frac{n}{2}}}
  \int_{0}^{1}\left(1-r^{2} \right)^{\frac{q}{p-1}}r^{n-1}F_{n} \ \dint r,
\end{equation*}
  where $F_{n}$ is the surface of the $n$ dimensional sphere with unit radius
\begin{equation*}
F_{n}=\frac{n\pi^{\frac{n}{2}}}{\Gamma\left(\frac{n}{2}+1 \right)}.
\end{equation*}
Evaluating the integral
\begin{equation*}
\int_{0}^{1}\left(1-r^{2} \right)^{\frac{q}{p-1}}r^{n-1}\ \dint r=
  \frac{\Gamma\left(\frac{q}{p-1}+1\right)\Gamma\left(\frac{n}{2}\right)}
  {2\Gamma\left(\frac{q}{p-1}+1+\frac{n}{2} \right)}
\end{equation*}
we have
\begin{equation}
\label{eq:intf_p^q1}
I=\left(\frac{a}{\pi}\right)^{\frac{n(q-1)}{2}}
  \left(\frac{\Gamma\left(\frac{p}{p-1}+\frac{n}{2} \right)}
             {\Gamma\left(\frac{p}{p-1}\right)}\right)^{q}
  \frac{\Gamma\left(\frac{q}{p-1}+1 \right)}{\Gamma\left(\frac{q}{p-1}+1+\frac{n}{2}\right)}
  (\det D)^{\frac{q-1}{2}}
\end{equation}
  and this verifies the Equation (\ref{eq:SR1_p(f_q)}).
If $p<1$ then the integral (\ref{eq:intf_p^q}) is
\begin{equation*}
I=\frac{A_{n,p}^{q}(\det D)^{\frac{q-1}{2}}}{a^{\frac{n}{2}}}
  \int_{0}^{1}\left(1+r^{2} \right)^{\frac{q}{p-1}}r^{n-1}F_{n} \ \dint r
\end{equation*}
  after the substitutions $a=\frac{1-p}{2p-n(p-1)}$ and $y_{i}=\sqrt{a\lambda_{i}}(x_{i}-u_{i})$.
Evaluating the integral we get
\begin{equation}
\label{eq:intf_p^q2}
I=\left(\frac{a}{\pi}\right)^{\frac{n(q-1)}{2}}
  \left(\frac{\Gamma\left(\frac{1}{1-p}\right)}
             {\Gamma\left(\frac{1}{1-p}-\frac{n}{2}\right)}\right)^{q}
  \frac{\Gamma\left(\frac{q}{1-p}-\frac{n}{2}\right)}{\Gamma\left(\frac{q}{1-p}\right)}
  (\det D)^{\frac{q-1}{2}}
\end{equation}
  which leads to Equation (\ref{eq:SR2_p(f_q)}).
\end{proof}

Since the $q$-R\'enyi entropy is independent of the expectation vector $\un{u}$ the entropy cannot
  generate a Riemannian metric on the whole space $\Xi$ just on $\Xi_{n}^{(s)}$.
The $q$-R\'enyi entropy can be written in the form of
\begin{equation}
\label{eq:q-Renyi altalanos}
S_{q}(f_{p}(D,\un{u},\cdot)=C_{n,p,q}-\frac{1}{2}\log\det D,
\end{equation}
so the quadratic form generated by the functional $S_{q}$ on the space of $p$-Gaussian
  distribution for every point $D\in\Xi_{n}^{(s)}$ and tangent vectors $X,Y\in\T_{n}$ being
\begin{equation*}
g^{(\RE)}_{D}(X,Y)=\left.\frac{\partial^{2}}{\partial s\partial t}
   S_{p}(f_{q}(D+tX+sY,\un{0},\cdot)\right\vert_{t=s=0}
  =-\frac{1}{2}\left.\frac{\partial^{2}}{\partial s\partial t}
  \bigl(\log\det(D+tX+sY)\bigr)\right\vert_{t=s=0}
\end{equation*}
is independent of $q$ and $p$.

\begin{theorem}
For every pair $(n,p)\in\mathcal{N}$ for every point $D\in\Xi_{n}^{(s)}$ and for every
  tangent vectors $X,Y\in\T_{n}$ we have
\begin{equation}
g^{(\RE)}_{D}(X,Y)=\frac{1}{2}\Tr(D^{-1}XD^{-1}Y)
\end{equation}
  for the quadratic form generated by the $q$-R\'enyi entropy.
\end{theorem}
\begin{proof}
To compute the derivative of the function $-\frac{1}{2}\log\det D$ we use the following
  equalities for symmetric strictly positive matrices
\begin{equation*}
\log\det A=\Tr\log A\qquad \log A=\int_{0}^{\infty}(E+\tau E)^{-1}-(A+\tau E)^{-1}\ \dint\tau,
\end{equation*}
  where $E$ denotes the identity  matrix.
Then the derivative is
\begin{align*}
& \left.\frac{\partial^{2}}{\partial s\partial t}-\frac{1}{2}\log\det(D+tX+sY)\right\vert_{t=s=0}
  =-\frac{1}{2}\left.\frac{\partial^{2}}{\partial s\partial t}\log(\det D)
  (\det(E+tXD^{-1}+sYD^{-1})\right\vert_{t=s=0}\\
& =-\frac{1}{2}\left.\frac{\partial^{2}}{\partial s\partial t}\log\det(E+tXD^{-1}+sYD^{-1}
  \right\vert_{t=s=0}
  =-\frac{1}{2}\left.\frac{\partial^{2}}{\partial s\partial t}\Tr\log(E+tXD^{-1}+sYD^{-1})
  \right\vert_{t=s=0}\\
&=-\frac{1}{2}\Tr\int_{0}^{\infty}\left.\frac{\partial^{2}}{\partial s\partial t} \left(
  (E+\tau E)^{-1}-(E+tXD^{-1}+sYD^{-1}+\tau E)^{-1}\right)\right\vert_{t=s=0}\ \dint\tau\\
&=\frac{1}{2}\Tr\int_{0}^{\infty} (E+\tau E)^{-1}
  \Bigl(YD^{-1}(E+\tau E)^{-1}XD^{-1}+XD^{-1}(E+\tau E)^{-1}YD^{-1}\Bigr)(E+\tau E)^{-1}
  \ \dint\tau\\
&=\Bigl(\Tr XD^{-1}YD^{-1}\Bigr)\int_{0}^{\infty}(1+\tau)^{-3}\ \dint\tau.
\end{align*}
This proofs the equality $g^{(\RE)}_{D}(X,Y)=\frac{1}{2}\Tr(D^{-1}XD^{-1}Y)$.
\end{proof}

The Tsallis entropy \cite{Tsa1} of the probability distribution $f$ is defined as
\begin{equation}
\label{eq:Tsallis entropy}
S^{(q)}(f)=\frac{1}{1-q}\left(\int_{\mathbb{R}}f(x)^{q}\ \dint x-1 \right)
\end{equation}
for parameter $q\in\mathbb{R}^{+}\setminus\left\lbrace 1\right\rbrace$.
Let us denote the quadratic form generated by the Tsallis entropy by $g^{(\EN,p,q)}$,
  i.e. for every point $D\in\Xi_{n}^{(s)}$ and tangent vectors $X,Y\in\T_{n}$
\begin{equation*}
g^{(\EN,p,q)}_{D}(X,Y)=\left.\frac{\partial^{2}}{\partial s\partial t}
 S^{(q)}(f_{p}(D+tX+sY,\un{0},\cdot)\right\vert_{t=s=0}
\end{equation*}
if $S^{(q)}(f_{p}(D+tX+sY,\un{0},\cdot)$ is well defined.

\begin{theorem}
For every pair $(n,p)\in\mathcal{N}$ for every point $D\in\Xi_{n}^{(s)}$ and for every tangent
  vectors $X,Y\in\T_{n}$ we have the following expressions for the quadratic form generated
  by the entropy
\begin{equation}
g^{(\EN,p,1)}_{D}(X,Y)=g^{(\EN,1,1)}_{D}(X,Y)=\frac{1}{2}\Tr(D^{-1}XD^{-1}Y).
\end{equation}
If $S^{(q)}(f_{p}(D,\un{u},\cdot))$ is well defined, then the generated quadratic form is
\begin{equation}
\label{eq:g^(EN,p,q)}
g^{(\EN,p,q)}_{D}(X,Y)=A_{n,p,q}'\det(D)^{\frac{q-1}{2}}
  \left(\Tr(D^{-1}XD^{-1}Y)-\frac{q-1}{2}\Tr XD^{-1}\Tr YD^{-1}\right),
\end{equation}
where
\begin{equation*}
A_{n,p,q}'=\left\lbrace\begin{array}{lll}
  \di\hskip-7pt\frac{1}{2}\hskip-2pt
  \left(\frac{p-1}{\pi(2p-n(p-1))}\right)^{\frac{n(q-1)}{2}}\hskip-5pt\times\hskip-4pt
  \left(\frac{\Gamma\left(\frac{p}{p-1}+\frac{n}{2} \right)}{\Gamma\left(\frac{p}{p-1}\right)}
  \right)^{q}\hskip-2pt
  \frac{\Gamma\left(\frac{q}{p-1}+1 \right)}{\Gamma\left(\frac{q}{p-1}+1+\frac{n}{2}\right)},
  &\hskip-5pt\mbox{if}&\di p>1,\ q>0;\\
  \di\hskip-7pt\frac{1}{2}\hskip-2pt
  \left(\frac{1-p}{\pi(2p-n(p-1))}\right)^{\frac{n(q-1)}{2}}\hskip-5pt\times\hskip-4pt
  \left(\frac{\Gamma\left(\frac{1}{1-p}\right)}{\Gamma\left(\frac{1}{1-p}-\frac{n}{2}\right)}
  \right)^{q}\hskip-2pt
  \frac{\Gamma\left(\frac{q}{1-p}-\frac{n}{2}\right)}{\Gamma\left(\frac{q}{1-p}\right)},
  &\hskip-15pt\mbox{if}&\di \hskip-10pt p<1,\ q>\frac{n(1-p)}{2}.
  \end{array}\right.
\end{equation*}
\end{theorem}
\begin{proof}
Since we have the limit
\begin{equation*}
\lim_{q\to 1}S^{(q)}(f_{p}(D,\un{u},\cdot))=\lim_{q\to 1}S_{q}(f_{p}(D,\un{u},\cdot))
  =S(f_{p}(D,\un{u},\cdot))
\end{equation*}
the formula $g^{(\EN,1,q)}_{D}(X,Y)=g^{(\EN,1,1)}_{D}(X,Y)$ is straightforward from the Equation
 (\ref{eq:q-Renyi altalanos}) and the metric was computed in the previous Theorem.

Now let us compute the derivative of the determinant function.
\begin{align*}
& \left.\frac{\dint}{\dint t}(\det(D+tX))\right\vert_{t=0}
  =\det D\left.\frac{\dint}{\dint t}\det(E+tXD^{-1}))\right\vert_{t=0}
  =\det D\left.\frac{\dint}{\dint t}(\exp\Tr\log(D+tX))\right\vert_{t=0}\\
& =\det D\Tr\left.\frac{\dint}{\dint t}\int_{0}^{\infty}
  (E+\tau E)^{-1}-(E+tXD^{-1}+\tau E)^{-1}\right\vert_{t=0}\ \dint\tau\\
& =\det D\Tr\int_{0}^{\infty}(E+\tau E)^{-1}XD^{-1}(E+\tau E)^{-1}\ \dint\tau
  =(\det D)(\Tr XD^{-1}).
\end{align*}
This can be expressed as $(\dint \det)(D)(X)=(\det D)(\Tr XD^{-1})$.
The second derivative of the determinant is
\begin{equation*}
(\dint^{2}\det)(D)(X)(Y)=(\det D)(\Tr XD^{-1})(\Tr YD^{-1})-(\det D)(\Tr XD^{-1}YD^{-1}).
\end{equation*}
According to these equalities the derivative of the function $(\det D)^{\frac{p-1}{2}}$ is
\begin{align*}
(\dint^{2} \det{}^{\frac{p-1}{2}})(D)(X)(Y)&=
  \left.\frac{\partial^{2}}{\partial s\partial t}(\det(D+tX+sY))^{\frac{p-1}{2}}
  \right\vert_{t=s=0}\\
&=\frac{p-1}{2}(\det D)^{\frac{p-1}{2}}\left(-\Tr XD^{-1}YD^{-1}
  +\frac{p-1}{2}\Tr XD^{-1}\Tr YD^{-1}\right).
\end{align*}
This gives the Equation (\ref{eq:g^(EN,p,q)}) and the constants $A'_{n,p,q}$ come from the
  Equations (\ref{eq:intf_p^q1},\ref{eq:intf_p^q2}).
\end{proof}

The Fisher information matrix is defined on parametric probability distributions.
At a point $(D,\un{u})\in\Xi_{n}$ the quadratic form
\begin{equation*}
g^{(\F,q)}_{D,\un{u}}((X,\un{x}),(Y,\un{y}))=\int_{\mathbb{R}^{n}}f_{q}(D,\un{u},\un{z})
  \frac{\partial\log f_{q}(D,\un{u},\un{z})}{\partial (X,\un{x})}
  \frac{\partial\log f_{q}(D,\un{u},\un{z})}{\partial (Y,\un{y})}\ \dint\un{z}
\end{equation*}
  (if the integral exists) gives rise to a positive definite matrix, so inducing a
  Riemannian metric on $\Xi_{n}$.
This metric is often called the expected information metric for the family of probability
  density functions; the original ideas are due to Fisher \cite{Fis} and Rao \cite{Rao}.

\begin{theorem}
For every pair $(n,p)\in\mathcal{N}$, where $p<2$ for every point $(D,\un{u})\in\Xi_{n}$
  and for every tangent vectors $(X,\un{x}),(Y,\un{y})\in\T_{n}$ the Fisher information matrix
  of $M_{p}$ is
\begin{align*}
g^{(\F,p)}_{D,\un{u}}((X,\un{x}),(Y,\un{y}))=&\frac{1}{2(2-p)}\Tr\bigl(D^{-1}XD^{-1}Y\bigr)+
  \frac{p-1}{4(2-p)}\Tr\bigl(D^{-1}X)\Tr\bigl(D^{-1}Y)\\
  &+\frac{2+n(p-1)}{(2p+n(p-1))(2-p)}\langle\un{x},D\un{y}\rangle.
\end{align*}
\end{theorem}
\begin{proof}
At a point $(D,\un{u})\in\Xi$ we have
\begin{equation}\label{eq:logfq}
\log f_{p}(D,\un{u},\un{x})=\log A_{n,p}+\frac{1}{2}\log\det{D}
  +\frac{1}{p-1}\log\left(1+\frac{1-p}{2p-n(1-p)}\langle\un{x}-\un{u},D(\un{x}-\un{u})
  \rangle \right).
\end{equation}
Choose our new coordinate system in $\mathbb{R}^{n}$ parallel to the eigenvectors of $D$.
In this coordinate system $D$ is diagonal, with entries $(\lambda_{i})_{i=1,\dots,n}$.
Let us denote by $(\un{e_{k}})_{k=1,\dots,n}$ the orthonormal basis.
According to Equation (\ref{eq:deriv_def}) the partial derivative of Equation (\ref{eq:logfq})
  with respect to a basis vector is
\begin{equation}
\label{eq:d log(f_q)/d e_k}
\frac{\partial \log f_{p}(D,\un{u},\un{x})}{\partial(0,\un{e_{k}}) }=
\frac{2}{2p-n(1-p)}\frac{1}{1+\frac{1-p}{2p-n(1-p)}\langle\un{x}-\un{u},D(\un{x}-\un{u})\rangle}
 \lambda_{k}(x_{k}-u_{k}).
\end{equation}
First we consider the $p\in\left\rbrack \frac{n}{n+2},1\right\lbrack$ case.
The Fisher information is
\begin{align*}
g^{(\F,p)}_{D,\un{u}}(&(0,\un{e_{k}}),(0,\un{e_{l}}))= \\
  &\frac{4\lambda_{k}\lambda_{l}A_{n,p}\sqrt{\det D}}{(2p-n(1-p))^{2}}\int_{\mathbb{R}^{n}}
 \left(1+\frac{1-p}{2p-n(1-p)}\langle\un{x}-\un{u},D(\un{x}-\un{u})
   \rangle \right)^{\frac{1}{p-1}-2}(x_{k}-u_{k})(x_{l}-u_{l})\ \dint\un{x}
\end{align*}
Introducing the new variables $a=\frac{1-p}{2p-n(1-p)}$, $y_{i}=\sqrt{a\lambda_{i}}(x_{i}-u_{i})$
  we have
\begin{equation*}
g^{(\F,p)}_{D,\un{u}}((0,\un{e_{k}}),(0,\un{e_{l}}))=\delta_{kl}\frac{A_{n,p}}{a^{\frac{n}{2}+1}}
  \frac{4\lambda_{k}}{(2p-n(1-p))^2}\int_{\mathbb{R}^{n}}\left(1+\sum_{i=1}^{n}y_{i}^{2}
  \right)^{\frac{1}{p-1}-2} y_{k}^{2}\dint \un{y}.
\end{equation*}
If $n=1$
\begin{equation*}
g^{(\F,p)}_{D,\un{u}}((0,\un{e_{1}}),(0,\un{e_{1}}))=\frac{A_{1,p}}{\sqrt{a}a}
  \frac{4\lambda_{1}}{(3p-1)^2}\int_{-\infty}^{\infty}\left(1+y^{2}\right)^{\frac{1}{p-1}-2}
  y^{2}\dint y
 =\lambda_{1}\frac{1+p}{(3p-1)(2-p)}
\end{equation*}
  and if $n>1$ then in the spherical coordinates in $n-1$ dimension we have
\begin{equation*}
g^{(\F,p)}_{D,\un{u}}((0,\un{e_{k}}),(0,\un{e_{l}}))=\delta_{kl}\frac{A_{n,p}}{a^{\frac{n}{2}+1}}
\frac{4\lambda_{k}}{(2p-n(1-p))^2}\int_{0}^{\infty}\int_{-\infty}^{\infty}
\left(1+y_{k}^{2}+r^{2}\right)^{\frac{1}{p-1}-2}
y_{k}^{2}r^{n-2}F_{n-1}\dint y_{k}\dint r.
\end{equation*}
Using the integral formulas
\begin{align*}
\int_{0}^{\infty}\int_{-\infty}^{\infty}
  \left(1+y_{k}^{2}+r^{2}\right)^{\frac{1}{p-1}-2}y_{k}^{2}r^{n-2}\dint y_{k}\dint r&=
  \frac{\sqrt{\pi}\Gamma\left(\frac{1}{1-p}+\frac{1}{2}\right)}
     {2\Gamma\left(\frac{1}{1-p}+2\right)}
  \int_{0}^{\infty}(1+r^{2})^{\frac{1}{p-1}-\frac{1}{2}}r^{n-2}\dint r\\
&=\frac{\Gamma\left(\frac{1}{1-p}+1-\frac{n}{2}\right)\Gamma\left(\frac{n-1}{2}\right)}
       {2\Gamma\left(\frac{1}{1-p}+\frac{1}{2}\right)}
\end{align*}
and after some simplification we have
\begin{equation}
\label{eq:q^(F,q) 1. part}
g^{(\F,p)}_{D,\un{u}}((0,\un{x}),(0,\un{y}))=
  \frac{2+n(p-1)}{(2p+n(p-1))(2-p)}\langle\un{x},D\un{y}\rangle
\end{equation}
  which is valid for every $n\in\mathbb{N}\setminus\left\lbrace 0\right\rbrace$.
Since $D$ is diagonal the partial derivative of the Equation (\ref{eq:logfq}) is
\begin{equation}
\label{eq:d log(f_q)/d E_ii}
\frac{\partial \log f_{p}(D,\un{u},\un{x})}{\partial(E_{ii},0) }=\frac{1}{2\lambda_{i}}
  +\frac{a}{p-1}\frac{(x_{i}-u_{i})^{2}}{1+a\langle\un{x}-\un{u},D(\un{x}-\un{u})\rangle}.
\end{equation}
The Fisher information is
\begin{align*}
g^{(\F,p)}_{D,\un{u}}((E_{ii},0),&(E_{kk},0))=
 \frac{1}{4}\frac{1}{\lambda_{i}\lambda_{k}}
  +\frac{a A_{n,p}\sqrt{\det D}}{2(p-1)\lambda_{k}}\int_{\mathbb{R}^{n}}
  \left(1+a\langle\un{x}-\un{u},D(\un{x}-\un{u})\rangle \right)^{\frac{1}{p-1}-1}
  (x_{i}-u_{i})^{2} \ \dint\un{x}\\
&+\frac{a A_{n,p}\sqrt{\det D}}{2(p-1)\lambda_{i}}\int_{\mathbb{R}^{n}}
  \left(1+a\langle\un{x}-\un{u},D(\un{x}-\un{u})\rangle \right)^{\frac{1}{p-1}-1}
  (x_{k}-u_{k})^{2} \ \dint\un{x}\\
&+\frac{a^{2} A_{n,p}\sqrt{\det D}}{(p-1)^{2}}\int_{\mathbb{R}^{n}}
  \left(1+a\langle\un{x}-\un{u},D(\un{x}-\un{u})\rangle \right)^{\frac{1}{p-1}-2}
  (x_{i}-u_{i})^{2}(x_{k}-u_{k})^{2} \ \dint\un{x}.
\end{align*}
The first integral is
\begin{equation*}
\frac{A_{n,p}}{2(p-1)a^{\frac{n}{2}}\lambda_{k}\lambda_{i}}
 \int_{0}^{\infty}\int_{-\infty}^{\infty} (1+y_{i}^{2}+r^2)^{\frac{1}{p-1}-1}
   y_{i}^{2}r^{n-2}F_{n-1}\ \dint y_{i}\dint r=-\frac{1}{4\lambda_{i}\lambda_{k}}.
\end{equation*}
The third one is if $i\neq k$
\begin{equation*}
\frac{A_{n,p}}{(p-1)^{2}a^{\frac{n}{2}}\lambda_{k}\lambda_{i}}
 \int_{0}^{\infty}\int_{-\infty}^{\infty}\int_{-\infty}^{\infty}
 (1+y_{i}^{2}+y_{k}^{2}+r^2)^{\frac{1}{p-1}-2} y_{i}^{2}y_{k}^{2}r^{n-3}F_{n-2}
  \ \dint y_{i}\ \dint y_{k}\dint r=\frac{1}{4(2-p)\lambda_{i}\lambda_{k}}
\end{equation*}
and if $i=k$
\begin{equation*}
\frac{A_{n,p}}{(p-1)^{2}a^{\frac{n}{2}}\lambda_{k}^{2}}
 \int_{0}^{\infty}\int_{-\infty}^{\infty}
 (1+y_{k}^{2}+r^2)^{\frac{1}{p-1}-2}y_{k}^{4}r^{n-2}F_{n-1}
 \ \dint y_{k}\dint r=\frac{3}{4(2-p)\lambda_{k}^{2}}.
\end{equation*}
Combining these integrals
\begin{equation*}
g^{(\F,p)}_{D,\un{u}}((E_{ii},0),(E_{kk},0))=\frac{1}{4}\left(\frac{1+2\delta_{ik}}{2-p}-1\right)
 \frac{1}{\lambda_{i}\lambda_{k}}.
\end{equation*}
If $D$ is not diagonal this can be expressed as
\begin{equation}
\label{eq:q^(F,q) 2. part}
g^{(\F,p)}_{D,\un{u}}((X,0),(Y,0))=\frac{1}{2(2-p)}\Tr\bigl(D^{-1}XD^{-1}Y\bigr)+
  \frac{p-1}{4(2-p)}\Tr\bigl(D^{-1}X)\Tr\bigl(D^{-1}Y).
\end{equation}
Finally the formulas (\ref{eq:q^(F,q) 1. part},\ref{eq:q^(F,q) 2. part}) give us the metric since
\begin{equation}
\label{eq:q^(F,q) 3. part}
g^{(\F,p)}_{D,\un{u}}((E_{ii},0),(0,\un{e_{k}}))=0.
\end{equation}
If $p>1$ then the partial derivatives given by the Equations
  (\ref{eq:d log(f_q)/d e_k},\ref{eq:d log(f_q)/d E_ii}) are the same.
The Fisher information is
\begin{equation*}
g^{(\F,p)}_{D,\un{u}}(0,\un{e_{k}}),(0,\un{e_{l}}))=
  \frac{4\lambda_{k}\lambda_{l}a^{2}A_{n,p}\sqrt{\det D}}{(p-1)^{2}}
  \hskip-8pt\int\limits_{\Dom(p,D,\un{u})}\hskip-8pt
 \left(1-a\langle\un{x}-\un{u},D(\un{x}-\un{u})\rangle \right)^{\frac{1}{p-1}-2}
 (x_{k}-u_{k})(x_{l}-u_{l})\ \dint\un{x},
\end{equation*}
  where $a=\dfrac{p-1}{2p-n(1-p)}$.
Introducing the new variables $y_{i}=\sqrt{a\lambda_{i}}(x_{i}-u_{i})$ we have
\begin{equation*}
g^{(\F,p)}_{D,\un{u}}((0,\un{e_{k}}),(0,\un{e_{l}}))=\delta_{kl}\frac{A_{n,p}}{a^{\frac{n}{2}}}
  \frac{4a\lambda_{k}}{(p-1)^2}\int_{B_{n}(1)}
  \left(1-\sum_{i=1}^{n}y_{i}^{2}\right)^{\frac{1}{p-1}-2}
  y_{k}^{2}\dint \un{y},
\end{equation*}
  where $B_{n}(1)$ is the closed unit ball in $\mathbb{R}^{n}$ with center origin.
If $n=1$
\begin{equation*}
g^{(\F,p)}_{D,\un{u}}((0,\un{e_{1}}),(0,\un{e_{1}}))=\frac{A_{1,p}}{\sqrt{a}}
\frac{4a\lambda_{1}}{(p-1)^2}\int_{-1}^{1}\left(1-y^{2}\right)^{\frac{1}{p-1}-2} y^{2}\dint y
 =\lambda_{1}\frac{1+p}{(3p-1)(2-p)}
\end{equation*}
  and if $n>1$ then in the spherical coordinates in $n-1$ dimension we have
\begin{equation*}
g^{(\F,p)}_{D,\un{u}}((0,\un{e_{k}}),(0,\un{e_{l}}))=\delta_{kl}\frac{A_{n,p}}{a^{\frac{n}{2}}}
\frac{4a\lambda_{k}}{(p-1)^2}\int_{0}^{1}\int\limits_{-\sqrt{1-r^{2}}}^{\sqrt{1-r^2}}
\left(1-y_{k}^{2}-r^{2}\right)^{\frac{1}{p-1}-2}
y_{k}^{2}r^{n-2}F_{n-1}\dint y_{k}\dint r.
\end{equation*}
Evaluating the integral
\begin{equation*}
\int_{0}^{1}\int\limits_{-\sqrt{1-r^{2}}}^{\sqrt{1-r^2}}\left(1-y_{k}^{2}-r^{2}
  \right)^{\frac{1}{p-1}-2} y_{k}^{2}r^{n-2}\dint y_{k}\dint r=
  \frac{1}{4}\frac{\sqrt{\pi}\Gamma\left(\frac{n-1}{2}\right)\Gamma\left(\frac{1}{p-1}-1\right)}
  {\Gamma\left(\frac{1}{p-1}+\frac{n}{2}\right)}
\end{equation*}
we get again Equation (\ref{eq:q^(F,q) 1. part}) for every $n$.
The Riemannian product of the matrix units is
\begin{align*}
g^{(\F,p)}_{D,\un{u}}((E_{ii},0),&(E_{kk},0))=
  \frac{1}{4}\frac{1}{\lambda_{i}\lambda_{k}}
  -\frac{a A_{n,p}\sqrt{\det D}}{2(p-1)\lambda_{k}}
  \hskip-8pt\int\limits_{\Dom(p,D,\un{u})}\hskip-8pt
  \left(1-a\langle\un{x}-\un{u},D(\un{x}-\un{u})\rangle \right)^{\frac{1}{p-1}-1}
  (x_{i}-u_{i})^{2} \ \dint\un{x}\\
&-\frac{a A_{n,p}\sqrt{\det D}}{2(p-1)\lambda_{i}}
  \hskip-8pt\int\limits_{\Dom(p,D,\un{u})}\hskip-8pt
  \left(1-a\langle\un{x}-\un{u},D(\un{x}-\un{u})\rangle \right)^{\frac{1}{p-1}-1}
  (x_{k}-u_{k})^{2} \ \dint\un{x}\\
&+\frac{a^{2} A_{n,p}\sqrt{\det D}}{(p-1)^{2}}
  \hskip-8pt\int\limits_{\Dom(p,D,\un{u})}\hskip-8pt
  \left(1-a\langle\un{x}-\un{u},D(\un{x}-\un{u})\rangle \right)^{\frac{1}{p-1}-2}
  (x_{i}-u_{i})^{2}(x_{k}-u_{k})^{2} \ \dint\un{x}.
\end{align*}
Introducing the variables $y_{i}=\sqrt{a\lambda_{i}}(x_{i}-u_{i})$ the domain $\Dom(p,D,\un{u})$
  will be transformed to $B_{n}(1)$.
Evaluating the integrals we get Equation (\ref{eq:q^(F,q) 2. part}).
Finally we note that Equation (\ref{eq:q^(F,q) 3. part}) is valid in this $p<1$ setting too,
  and this completes the proof.
\end{proof}

In the $p\geq 2$ case the Fisher information does not exist, since the integral which defines it
  divergent.

Relative entropies are special distance functions between probability measures and although
  there are several relative entropy functions, but most of them are special Csisz\'ar
  $\varphi$-divergences \cite{Csis1,Csis2}.
Assume that $\varphi:\mathbb{R}^{+}\to\mathbb{R}$ is a strictly convex function, and
  $\varphi(1)=0$.
Then one can define the Csisz\'ar $\varphi$-relative entropy as
\begin{equation*}
H^{(\varphi,p)}(f_{p}(D_{1},\un{u}_{1},\cdot),f_{p}(D_{2},\un{u}_{2},\cdot))
  =\int_{\mathbb{R}^{n}}f_{p}(D_{1},\un{u}_{1},\un{x})
  \varphi\left(\frac{f_{p}(D_{2},\un{u}_{2},\un{x})}{f_{p}(D_{1},\un{u}_{1},\un{x})}\right)\
  \dint\un{x}.
\end{equation*}
For example the Kullback--Liebler \cite{LeiKul}, Hellinger \cite{Hel} and $\alpha$-relative
  entropies are given by the functions $\varphi(x)=-\log x$, $\varphi(x)=(1-\sqrt{x})^{2}$ and
  $\varphi(x)=\dfrac{4}{1-\alpha^2}\left(1-x^{\frac{1+\alpha}{2}} \right)$.
We note that the $\alpha$-relative entropy is strongly related to the R\'enyi \cite{Ren} and
  to the Tsallis entropy \cite{Tsa1,Tsa2}.
The quadratic form induced by the $\varphi$-divergence is
\begin{equation*}
g^{(\varphi,p)}_{D,\un{u}}((X,\un{x}),(Y,\un{y}))=\left.\frac{\partial^{2}}{\partial s\partial t}
 H^{(\varphi)}(f_{p}(D,\un{u},\cdot),f_{p}(D+tX+sY,\un{u}+t\un{x}+s\un{y},\cdot))
 \right\vert_{t=s=0}.
\end{equation*}

\begin{theorem}
Assume that $\varphi:\mathbb{R}^{+}\to\mathbb{R}$ is a strictly convex function
  and $\varphi(1)=0$.
Then $g^{(\varphi,p)}=\varphi''(1) g^{(\F,p)}$ on the manifold $\Xi$.
\end{theorem}
\begin{proof}
The computation
\begin{align*}
g^{(\varphi,p)}_{D,\un{u}}((X,\un{x}),(Y,\un{y}))&=\left.\frac{\partial^{2}}{\partial s\partial t}
  H^{(\varphi)}(f_{p}(D,\un{u},\cdot),f_{p}(D+tX+sY,\un{u}+t\un{x}+s\un{y},\cdot))
  \right\vert_{t=s=0}\\
&=\int_{\mathbb{R}^{n}}f_{p}(D,\un{u},\un{z})
  \left\lbrack\left.\frac{\partial^{2}}{\partial s\partial t}
  \varphi\left(\frac{f_{p}(D+tX+sY,\un{u}+t\un{x}+s\un{y},\un{z})}{f_{p}(D,\un{u},\un{z})}\right)
  \right\vert_{t=s=0}\right\rbrack\ \dint\un{z}\\
&=\int_{\mathbb{R}^{n}}\frac{\varphi''(1)}{f_{p}(D,\un{u},\un{z})}
\cdot\left.\frac{\dint f_{p}(D+tX,\un{u}+t\un{x},\un{z})}{\dint t}\right\vert_{t=0}
\cdot\left.\frac{\dint f_{p}(D+sY,\un{u}+s\un{y},\un{z})}{\dint s}\right\vert_{s=0}\ \dint\un{z}\\
&=\varphi''(1)\int_{\mathbb{R}^{n}}\frac{1}{f_{p}(D,\un{u},\un{z})}
 \cdot\frac{\partial f_{p}(D,\un{u},\un{z})}{\partial (X,\un{x})}
 \cdot\frac{\partial f_{p}(D,\un{u},\un{z})}{\partial (Y,\un{y})}  \ \dint\un{z}\\
&=\varphi''(1)g^{(\F,p)}_{D,\un{u}}((X,\un{x}),(Y,\un{y}))
\end{align*}
verifies the Theorem.
\end{proof}

Calvo and Oller studied a different metric on the space $\Xi_{n}$ \cite{CalOll}.
Their starting point was the metric
\begin{equation*}
g:P(n)\times M_{n}\times M_{n}\to\mathbb{R}\quad (D,X,Y)\mapsto \frac{1}{2}\Tr(D^{-1}XD^{-1}Y),
\end{equation*}
  where $P(n)$ denotes the set of $n\times n$ real, symmetric, positive definite matrices,
  and the embedding
\begin{equation*}
\pi:\MN\times\mathbb{R}^{n+1}\times\mathbb{R}^{+}\to P(n)\quad (K,\un{u},\beta)\mapsto
 \begin{pmatrix}
 K+\beta \un{u}\circ\un{u}& \beta\un{u}\\
 \beta\un{u} & \beta\end{pmatrix}.
\end{equation*}
The metric $g$ has been studied by Siegel \cite{Sie}, James \cite{Jam} and Burbea \cite{Bur}.
Calvo and Oller considered the pull-back metric of $g$ by $\pi$ restricted to the manifold
  $\MN\times\mathbb{R}^{n}\times\left\lbrace \beta\right\rbrace$, which is
\begin{equation*}
\tilde{g}_{\pi}\vert_{\MN\times\mathbb{R}^{n}\times\left\lbrace \beta\right\rbrace}
  :(\MN\times\mathbb{R}^{n})\times\T_{n}\times\T_{n}
  \to\mathbb{R}\quad (K,\un{u}),(X,\un{x}),(Y,\un{y})\mapsto
  \frac{1}{2}\Tr(K^{-1}XK^{-1}Y)+\beta\langle\un{x},K^{-1}\un{y}\rangle.
\end{equation*}
If we use our parametrization of the normal distributions, namely the inverse of the
  covariance matrix and the expectation vector, then the metric is
\begin{equation*}
g^{(\SI,\beta)}:\Xi_{n}\times\T_{n}\times\T_{n}
  \to\mathbb{R}\quad (D,\un{u}),(X,\un{x}),(Y,\un{y})\mapsto
  \frac{1}{2}\Tr(D^{-1}XD^{-1}Y)+\beta\langle\un{x},D\un{y}\rangle.
\end{equation*}

Lovri\'c, Min-Oo and Ruh \cite{LovMinRuh} studied a slightly different metric on the
  space $\Xi_{n}$.
Let us sketch their fundamental idea briefly.
Denote by $P_{1}(n)$ the set of $n\times n$ real, symmetric, positive definite matrices
  with determinant $1$.
Then the map
\begin{equation*}
j:\MN\times\mathbb{R}^{n}\to P_{1}(n+1)\quad (K,\un{u})\mapsto (\det K)^{-\frac{2}{n+1}}
  \begin{pmatrix}
  K^{2}+\un{u}\circ\un{u} & \un{u} \\
  \un{u}^{T} & 1 \end{pmatrix}
\end{equation*}
  is a smooth bijection.
The special linear group has a natural smooth group action on $P_{1}(n)$
\begin{equation*}
A:\SL(n)\to\Aut(P_{1}(n))\quad g\mapsto\Bigl(m\mapsto gmg^{T} \Bigr),
\end{equation*}
  where $g^{T}$ denotes the transpose of $g$.
This group action represents $P(n)$ as the Riemannian symmetric space $\SL(n)/\SO(n)$ with
  $\SO(n)$ principal bundle
\begin{equation*}
\SL(n)\to P(n)\quad g\mapsto gg^{T}.
\end{equation*}
This means that the space $\MN\times\mathbb{R}^{n}$ can be considered as a Riemannian
  symmetric space $\SL(n+1)/\SO(n+1)$.
It is known in the theory of symmetric spaces, that the natural $\SL(n+1)$ invariant metric
  on the space $P_{1}(n+1)$ is given by restricting the Killing form of the simple Lie algebra
  $\mathop{\textrm{sl}}\nolimits(n+1)$ to the subspace $\mathfrak{s}_{0}(n+1)$ under
  the Cartan decomposition
  $\mathop{\textrm{sl}}\nolimits(n+1)=\mathfrak{o}(n+1)\oplus\mathfrak{s}_{0}(n+1)$.
The generated metric is unique up to a positive constant factor.
This metric at a point $(K,\un{u})\in\MN\times\mathbb{R}^{n}$ for tangent vectors
  $(X,\un{x}),(Y,\un{y})\in\T_{n}$
  is given by the equation
\begin{equation*}
g^{(\mbox{inv.})}_{K,\un{u}}((X,\un{x}),(Y,\un{y}))
=\Tr\bigl(K^{-1}XK^{-1}Y\bigr)-\frac{1}{n+1}(\Tr K^{-1}X)(\Tr K^{-1}Y)
  +\frac{1}{2}\langle\un{x},K^{-1}\un{y}\rangle.
\end{equation*}
Using the inverse of the covariance matrix as a parameter, this metric is
\begin{equation*}
g^{(\LMR)}_{D,\un{u}}((X,\un{x}),(Y,\un{y}))
=\Tr\bigl(D^{-1}XD^{-1}Y\bigr)-\frac{1}{n+1}(\Tr D^{-1}X)(\Tr D^{-1}Y)
  +\frac{1}{2}\langle\un{x},D\un{y}\rangle.
\end{equation*}

\begin{corollary}
On the parameter space of the special normal distributions we have the equality of the metrics
\begin{equation*}
g^{(\RE)}=g^{(\T,p,1)}=g^{(\F,1)}\vert_{\Xi_{n}^{(s)}}=g^{(\SI,\beta)}\vert_{\Xi_{n}^{(s)}}.
\end{equation*}
On the parameter space of the normal distributions we have
\begin{equation*}
g^{(\F,1)}=g^{(\SI,1)},
\end{equation*}
  but the metrics $g^{(\F,p)}$, $g^{(\SI,\beta)}$ and $g^{(\LMR)}$ are pairwise incomparable
  in the sense that there is no $(n,p)\in\mathcal{N}$ parameter such that two of these metrics
  are equal up to a multiplicative factor.
\end{corollary}

To work with the Riemannian metrics $g^{(\RE)}$, $g^{(\EN,1,q)}$, $g^{(\F,q)}$, $g^{(\varphi,q)}$,
  $g^{(\SI,\beta)}$ and $g^{(\LMR)}$ simultaneously we consider the metric
\begin{equation*}
g_{D,\un{u}}((X,\un{x}),(Y,\un{y}))=\frac{1}{2}\Tr\bigl(D^{-1}XD^{-1}Y\bigr)
  +\alpha(\Tr D^{-1}X)(\Tr D^{-1}Y)+\beta\langle\un{x},D\un{y}\rangle
\end{equation*}
  with parameters  $\alpha,\beta\in\mathbb{R}$, $\alpha\neq-\frac{1}{2n}$, $\beta\neq 0$ on
  the manifold $\Xi_{n}$.
In the $\beta=0$ case we restrict the manifold to $\Xi_{n}^{(s)}$.
For every $D\in\Xi_{n}^{(s)}$ and $X,Y\in\T_{n}$ we have a Cauchy--Schwarz inequality
\begin{equation*}
\Tr^{2}\bigl(D^{-1}XD^{-1}Y\bigr)\leq \Tr\bigl(D^{-1}XD^{-1}X\bigr)\Tr\bigl(D^{-1}YD^{-1}Y\bigr).
\end{equation*}
Substituting $Y=D$ we have
\begin{equation*}
\frac{1}{2}\Tr\bigl(D^{-1}XD^{-1}X\bigr)-\frac{1}{2n}(\Tr D^{-1}X)^{2}\geq 0.
\end{equation*}
It means that if $\alpha>-\frac{1}{2n}$ then $g$ is a Riemannian metric,
  if $\alpha<-\frac{1}{2n}$ then $g$ is a semi-Riemannian metric, and in the
  $\alpha=-\frac{1}{2n}$ case $g$ is a degenerated quadratic form.
The Theorems and proofs are valid for semi-Riemannian metrics too,
  so we have just one condition $\alpha\neq-\frac{1}{2n}$.

\section{Geodesics}

In this section we derive the differential equation of the geodesic lines in the space $M_{n}$
  and we present some solutions.

\begin{theorem}
A curve $\gamma:\mathbb{R}\to\Xi_{n}$, $\gamma(t)=(D(t),\un{u}(t))$ is a geodesic curve if
  and only if for every $t\in\Dom\gamma$
\begin{align}
\label{eq:geodesic line}
\ddot D(t)&=\dot D(t)D(t)^{-1}\dot D(t)+\beta (D(t)\dot{\un{u}}(t))\circ (D(t)\dot{\un{u}}(t))
  -\frac{2\alpha\beta}{1+2n\alpha}\langle \dot{\un{u}}(t),D(t)\dot{\un{u}}(t)\rangle D(t)\\
\ddot{\un{u}}(t)&=-D(t)^{-1}\dot D(t)\dot{\un{u}}(t)\nonumber
\end{align}
holds.
\end{theorem}
\begin{proof}
Denote by $\GL(n)$ the set of invertible $n\times n$ matrices and define the reciprocal function as
\begin{equation*}
i:\GL(n)\to \GL(n)\qquad D\mapsto D^{-1}.
\end{equation*}
At the point $D$ the tangent space $\T_{D}\GL(n)$ can be identified with the set of $n\times n$
  matrices $\Mat(n)$.
The derivative of the inversion function is
\begin{equation*}
\dint i:\GL(n)\to\Lin(\Mat(n),\Mat(n))\qquad (D)\mapsto (\dint i)(D)\ = \
\Bigl(A\mapsto (\dint i)(D)(A)=-D^{-1}AD^{-1} \Bigr).
\end{equation*}
This leads to the derivative of the metric
\begin{align*}
\dint g&:\Xi_{n}\to\Lin(\T_{n},\Lin(\T_{n}\times\T_{n},\mathbb{R}))\\
&(D,\un{u})\mapsto\Bigl( (Z,\un{z})\mapsto\bigl( ((X,\un{x}),(Y,\un{y}))\mapsto
 \dint g_{D,\un{u}}(Z,\un{z})((X,\un{x}),(Y,\un{y}))   \bigr)   \Bigr),
\end{align*}
where
\begin{align*}
\dint g_{D,\un{u}}(Z,\un{z})((X,\un{x}),(Y,\un{y}))&=-\frac{1}{2}\Tr D^{-1}(ZD^{-1}X+XD^{-1}Z)
  D^{-1}Y +\beta\langle\un{x},D\un{y}\rangle\\
&-\alpha\Tr(D^{-1}ZD^{-1}X)\Tr(D^{-1}Y)
 -\alpha\Tr(D^{-1}X)\Tr(D^{-1}ZD^{-1}Y).
\end{align*}
At a given point $(D,\un{u})\in\Xi_{n}$ for given tangent vectors
  $(X,\un{x}),(Y,\un{y})\in\T_{n}$ the map
\begin{align*}
\tau&_{(D,\un{u}),(X,\un{x}),(Y,\un{y})}:\T_{n}\to\mathbb{R}\\
&(Z,\un{z})\mapsto\frac{1}{2}\Bigl(
  \dint g_{D,\un{u}}(Y,\un{y})((X,\un{x}),(Z,\un{z}))+
  \dint g_{D,\un{u}}(X,\un{x})((Y,\un{y}),(Z,\un{z}))-
  \dint g_{D,\un{u}}(Z,\un{z})((X,\un{x}),(Y,\un{y}))  \Bigr)
\end{align*}
  is a linear functional.
It means that there exists a unique tangent vector $V_{(D,\un{u}),(X,\un{x}),(Y,\un{y})}\in\T_{n}$
  such that for all vectors $(Z,\un{z})\in\T_{n}$
\begin{equation*}
g_{D,\un{u}}(V_{(D,\un{u}),(X,\un{x}),(Y,\un{y})},(Z,\un{z}))=
\tau_{(D,\un{u}),(X,\un{x}),(Y,\un{y})}(Z,\un{z})
\end{equation*}
  holds.
One can define the map
\begin{equation*}
\Gamma:\Xi_{n}\to\Lin(\T_{n}\times\T_{n},\T_{n})\qquad (D,\un{u})\mapsto
\Bigl( ((X,\un{x}),(Y,\un{y}))\mapsto V_{(D,\un{u}),(X,\un{x}),(Y,\un{y})} \Bigr)
\end{equation*}
  which is called covariant derivative.
It means, that the equation for all tangent vectors $(Z,\un{z})\in\T_{n}$
\begin{align}
\label{eq:covariant derivarive find 1}
g_{D,\un{u}}(\Gamma_{(D,\un{u})}(X,\un{x})(Y,\un{y}),(Z,\un{z}))=
  &-\frac{1}{4}\Tr\Bigl(D^{-1}(XD^{-1}Y+YD^{-1}X)D^{-1}Z \Bigr)\\
  &-\frac{\alpha}{2}\Tr(D^{-1}(XD^{-1}Y+YD^{-1}Z))\Tr(D^{-1}Z)\nonumber\\
  &+\frac{\beta}{2}(\langle\un{y},X\un{z}\rangle+\langle\un{x},Y\un{z}\rangle
  -\langle\un{x},Z\un{y}\rangle )\nonumber
\end{align}
determines the covariant derivative.
Let us write the covariant derivative in the form of
\begin{equation*}
\Gamma_{(D,\un{u})}(X,\un{x})(Y,\un{y})=\left(
  -\frac{1}{2}(XD^{-1}Y+YD^{-1}X)+DW,D^{-1}\un{w}\right)
\end{equation*}
  for some $W\in\T_{n}$ and $\un{w}\in\mathbb{R}^{n}$.
Then Equation (\ref{eq:covariant derivarive find 1}) is
\begin{equation}
\label{eq:geodesic line proof01}
\frac{1}{2}\Tr(WD^{-1}Z)+\alpha\Tr W\Tr D^{-1}Z+\beta\langle\un{w},\un{z}\rangle=
  \frac{\beta}{2}(\langle\un{y},X\un{z}\rangle+\langle\un{x},Y\un{z}\rangle
  -\langle\un{x},Z\un{y}\rangle ).
\end{equation}
Let us introduce the notation $\odot$ for symmetrized diadic product, for vectors
  $\un{u},\un{v}\in\mathbb{R}^{n}$
\begin{equation*}
\un{u}\odot\un{v}=\un{u}\circ\un{v}+\un{v}\circ\un{u},
\end{equation*}
  that is, the components of the $n\times n$ matrix $(\un{u}\odot\un{v})$ are
  $(\un{u}\odot\un{v})_{ij}=\un{u}_{i}\un{v}_{j}+\un{u}_{j}\un{v}_{i}$.
Equation (\ref{eq:geodesic line proof01}) means that the vector component of the covariant derivative is
\begin{equation*}
\un{w}=\frac{1}{2}(X\un{y}+Y\un{x})
\end{equation*}
  and the remaining matrix part is
\begin{equation*}
W=-\frac{\beta}{2}(\un{x}\odot\un{y})D+\gamma E
\end{equation*}
  where the parameter $\gamma$ can easily be found.
Combining the terms together we have the following expression for the covariant derivative.
\begin{equation}
\label{eq:covariant derivative}
\Gamma_{(D,\un{u})}(X,\un{x})(Y,\un{y})\hskip-1pt=\hskip-3pt\left(\hskip-3pt
  -\frac{1}{2}(XD^{-1}Y+YD^{-1}X)
  \hskip-2pt-\hskip-2pt\frac{\beta}{2}D(\un{x}\odot\un{y})D
  \hskip-2pt+\hskip-2pt\frac{2\alpha\beta}{1+2n\alpha}\langle\un{x},D\un{y}\rangle D,
  \frac{1}{2}D^{-1}(X\un{y}+Y\un{x})\right)
\end{equation}
A curve $\gamma:\mathbb{R}\to\Xi$ is called a geodesic curve if
\begin{equation*}
\forall t\in\Dom(\gamma):\quad \ddot\gamma(t)+\Gamma_{\gamma(t))}(\dot\gamma(t))(\dot\gamma(t))=0
\end{equation*}
  holds.
Consider a curve $\gamma:\mathbb{R}\to\Xi$, $\gamma(t)=(D(t),\un{u}(t))$,
  substitute it into the equation of the geodesic curve and expand the covariant derivative, then
  according to Equation (\ref{eq:covariant derivative}) we get Equation (\ref{eq:geodesic line})
  of the Theorem.
\end{proof}

We have some remarks about the geodesic curves, which are only valid  for Riemannian metrics.

\begin{remark}
Let us consider the case $n=1$, and assume that $\alpha\neq-\frac{1}{2n}$.
Then the system of differential equations of the geodesic line is
\begin{equation*}
\ddot D(t)=\frac{\dot D(t)^{2}}{D(t)}+\frac{\beta}{1+2\alpha}D(t)^{2}\dot u(t)^{2}\qquad
\ddot u(t)=-\frac{\dot D(t)}{D(t)}\dot u(t).
\end{equation*}
The curve $\gamma:\mathbb{R}\to\Xi_{1}$
\begin{equation*}
\gamma(t)=\left(\frac{2}{a^{2}}\cosh^{2}(bt+c),a\sqrt{\frac{1+2\alpha}{\beta}}\tanh(ct+b)+d\right)
\end{equation*}
is a geodesic line.

Assume that we have two points $(D_{0},\un{u}_{0}),(D_{1},\un{u}_{1})$ in the space $\Xi_{1}$
  and assume that $\un{u}_{1}>\un{u}_{0}$.
Let us define the following quantities
\begin{align*}
x&=(\un{u}_{1}-\un{u}_{0})\sqrt{\frac{D_{0}\beta}{2+4\alpha}}
\qquad y=\sqrt{\frac{D_{1}}{D_{0}}}
\qquad c=\log\left(
  \frac{\sqrt{(x^{2}y^{2}+1-y^{2})^{2}+4x^{2}y^{4}}-(x^{2}y^{2}+1-y^{2})}{2xy^{2}}\right)\\
a&=\sqrt{\frac{2}{D_{0}}}\cosh c
\qquad b=-\log\left(y(1-x\ce^{c}) \right)
\qquad d=\un{u}_{0}-a\sqrt{\frac{1+2\alpha}{\beta}}\tanh c.
\end{align*}
Then the curve $\gamma:\left\lbrack 0,1\right\rbrack\to\Xi_{1}$
\begin{equation*}
\gamma(t)=\left(\frac{2}{a^{2}}\cosh^{2}(bt+c),a\sqrt{\frac{1+2\alpha}{\beta}}\tanh(bt+c)+d\right)
\end{equation*}
is a geodesic line, such that $\gamma(0)=(\un{u}_{0},D_{0})$ and $\gamma(1)=(\un{u}_{1},D_{1})$.
Simple calculation shows that the distance between the points
  $(\un{u}_{0},D_{0}),(\un{u}_{1},D_{1})\in\Xi_{1}$ is
\begin{align*}
d\Bigl((\un{u}_{0},D_{0}),(\un{u}_{1},D_{1})\Bigr)=\int_{0}^{1}
  \sqrt{g_{\gamma(t)}(\dot\gamma(t),\dot\gamma(t))}\ \dint t
 =\sqrt{2+4\alpha}\vert b\vert.
\end{align*}
\end{remark}

The geodesic line and the Rao distance on the space of special normal distributions has been
  computed by Siegel \cite{Sie} and Burbea \cite{Bur}.
The next Remark concerns their results.

\begin{remark}
Let us consider the space of $n$ dimensional special normal distributions $\Xi^{(s)}_{n}$
  and assume that $\alpha\neq-\frac{1}{2n}$.
Then the curve $D:\mathbb{R}\to\Xi^{(s)}_{n}$ is a geodesic line if and only if
\begin{equation*}
\ddot D(t)=\dot{D}(t)D(t)^{-1}\dot{D}(t).
\end{equation*}
Assume that we have two points $D_{0},D_{1}$ in the space $\Xi^{(s)}_{n}$.
Then the curve
\begin{equation}
\label{eq:Siegel geodesic line}
\gamma:\left\lbrack 0,1\right\rbrack\to\Xi \qquad  t\mapsto  D_{0}^{\frac{1}{2}}
  \exp\left(t\log\left(D_{0}^{-\frac{1}{2}}D_{1}D_{0}^{-\frac{1}{2}}\right)\right)
  D_{0}^{\frac{1}{2}}
\end{equation}
is a geodesic line, such that $\gamma(0)=D_{0}$ and $\gamma(1)=D_{1}$.
The distance between the points $D_{0},D_{1}\in\Xi^{(s)}_{n}$ is
\begin{align*}
d(D_{0},D_{1})
  &=\int_{0}^{1}\sqrt{g_{\gamma(t)}(\dot\gamma(t),\dot\gamma(t))}\ \dint t
  =\sqrt{\frac{1}{2}\Tr\log^{2}\left(D_{0}^{-\frac{1}{2}}D_{1}D_{0}^{-\frac{1}{2}}\right)
  +\alpha\Tr^{2}\log\left(D_{0}^{-\frac{1}{2}}D_{1}D_{0}^{-\frac{1}{2}}\right)
  }.
\end{align*}
\end{remark}

\begin{remark}
In the $\alpha=0$ case the metric is the pull-back of the Siegel metric by the embedding
\begin{equation*}
\pi_{\beta}:\Xi_{n}\to P(n+1)\quad (D,\un{u})\mapsto
  \begin{pmatrix}
  D^{-1}+\beta \un{u}\circ\un{u}& \beta\un{u}\\
  \beta\un{u} & \beta\end{pmatrix}.
\end{equation*}
The equation of the geodesic line in the Siegel metric is given by Equation
  (\ref{eq:Siegel geodesic line}).
If we have two points $(D_{0},\un{u}_{0}),(D_{1},\un{u}_{1})\in\Xi_{n}$ then we define the
  matrices $S_{i}=\pi_{\beta}(D_{i},\un{u}_{i})$ ($i=0,1$), the equation of the geodesic line is
\begin{equation*}
\gamma:\left\lbrack 0,1\right\rbrack\to\Xi \qquad  t\mapsto \pi_{\beta}^{-1}\left\lbrack
  S_{0}^{\frac{1}{2}}\exp\left(t\log\left(S_{0}^{-\frac{1}{2}}S_{1}S_{0}^{-\frac{1}{2}}
  \right)\right)S_{0}^{\frac{1}{2}}\right\rbrack
\end{equation*}
  and the distance between the points is
\begin{equation*}
d((D_{0}\un{u_{0}}),(D_{1},\un{u_{1}}))
    =\sqrt{\frac{1}{2}\Tr\log^{2}\left(S_{0}^{-\frac{1}{2}}S_{1}S_{0}^{-\frac{1}{2}}\right)}.
\end{equation*}

\end{remark}

\begin{remark}
In the $\alpha=0$ case we can more exact parametrization of the geodesic line in some special
  cases.
Assume that $B,C\in\MN$ are diagonal matrices,
  $\un{A},\un{D}\in\mathbb{R}^{n}$ are vectors such that the components of $\un{A}$ are equal
  and $U$ is an $n\times n$ orthogonal matrix such that $U\un{A}=\un{A}$.
Then the curve
\begin{equation*}
\gamma:\mathbb{R}^{+}\to\Xi\qquad
  t\mapsto\left(\frac{2}{\Vert \un{A}\Vert^{2}}U\cosh^{2}(Bt+C)U^{-1},
  \sqrt{\frac{n}{\beta}}U\tanh(Bt+C)\un{A}+\un{D} \right)
\end{equation*}
is a geodesic line.
The distance between the points $\gamma(t_{0})$ and $\gamma(t_{1})$
  ($t_{0},t_{1}\in\mathbb{R}^{+}$) is
\begin{align*}
d(\gamma(t_{0}),\gamma(t_{1}))&=\int_{t_{0}}^{t_{1}}
  \sqrt{g_{\gamma(t)}(\dot\gamma(t),\dot\gamma(t))}\ \dint t
  =\vert t_{1}-t_{0}\vert \sqrt{2\Tr B^{2}}.
\end{align*}
\end{remark}

\section{Curvatures}

Since Efron clarified the statistical meaning of the curvature, different curvature tensors
  has been studied on statistical manifolds.
For curvatures on the space of normal distributions  see for example Amari \cite{Ama2,Ama1},
  Siegel \cite{Sie}, Burbea \cite{Bur}, Skovgaard \cite{Sko}, Calvo and Oller \cite{CalOll},
  Lovri\'c Min-Oo and Ruh \cite{LovMinRuh}.

\begin{theorem}
For every point $(D,\un{u})\in\Xi_{n}$ and for every tangent vectors
  $(X,\un{x}),(Y,\un{y}),(Z,\un{z})\in\T_{n}$ the Riemann curvature tensor is
\begin{align}
&R_{(D,\un{u})}((X,\un{x}),(Y,\un{y}),(Z,\un{z}))\hskip-2pt=\hskip-4pt\Bigg\lbrack\hskip-2pt
\frac{1}{4}\Bigl(ZD^{-1}XD^{-1}Y+YD^{-1}XD^{-1}Z-XD^{-1}YD^{-1}Z-ZD^{-1}YD^{-1}X\Bigr)\nonumber\\
&+\hskip-2pt\frac{\beta}{4}\Bigl(Y\un{x}\odot D\un{z}-X\un{y}\odot D\un{z}
  +D\un{y}\odot Z\un{x}-D\un{x}\odot Z\un{y}\Bigr)
  \hskip-2pt+\hskip-2pt\frac{\alpha\beta}{1+2n\alpha}\Bigl(\langle X\un{y}-Y\un{x},\un{z}\rangle
  +\langle \un{x},Z\un{y}\rangle-\langle\un{y},Z\un{x}\rangle \Bigr)D,  \nonumber\\
& \frac{1}{4}\Bigl(D^{-1}(YD^{-1}X-XD^{-1}Y)\un{z}+D^{-1}ZD^{-1}(X\un{y}-Y\un{x})\Bigr)
 +\frac{\beta}{4}\Bigl(\un{z}\odot\un{x})D\un{y}-(\un{z}\odot\un{y})D\un{x}\Bigr)    \nonumber\\
&+\frac{\alpha\beta}{1+2n\alpha}\Bigl(\langle\un{y},D\un{z}\rangle\un{x}
  -\langle\un{x},D\un{z}\rangle\un{y} \Bigr)\Bigg\rbrack.
\label{eq:Riemann curvature tensor}
\end{align}
\end{theorem}
\begin{proof}
The derivative of the covariant derivative is
\begin{align*}
\dint\Gamma:&\Xi_{n}\to\Lin(\T_{n},\Lin(\T_{n}\times\T_{n},\T_{n}))\\
&(D,\un{u})\mapsto\Bigg((X,\un{x})\mapsto\bigl(((Y,\un{y}),(Z,\un{z}))\mapsto
\dint\Gamma_{(D,\un{u})}(X,\un{x})(Y,\un{y})(Z,\un{z}) \bigr)  \Bigg)
\end{align*}
where from Equation (\ref{eq:covariant derivative})
\begin{align*}
\dint\Gamma_{(D,\un{u})}(Z,\un{z})(X,\un{x})(Y,\un{y})=&\frac{1}{2}(XD^{-1}ZD^{-1}Y
  +YD^{-1}ZD^{-1}X)-\frac{\beta}{2}\Bigl(Z(\un{x}\odot\un{y})D+D(\un{x}\odot\un{y})Z\Bigr)\\
&+\frac{2\alpha\beta}{1+2n\alpha}\Bigl(\langle\un{x},Z\un{y}\rangle D
  +\langle\un{x},D\un{y}\rangle Z\Bigr)-\frac{1}{2}D^{-1}ZD^{-1}(X\un{y}+Y\un{x}).
\end{align*}
The Riemann curvature tensor is defined to be
\begin{equation*}
R:\Xi_{n}\to\Lin(\T_{n}\times\T_{n}\times\T_{n},\T_{n})\quad
(D,\un{u})\mapsto\Bigl( ((X,\un{x})(Y,\un{y})(Z,\un{z}))\mapsto
R_{(D,\un{u})}((X,\un{x}),(Y,\un{y}),(Z,\un{z}))\Bigr),
\end{equation*}
where
\begin{align*}
R_{(D,\un{u})}((X,\un{x}),(Y,\un{y}),(Z,\un{z}))=&
  \dint\Gamma_{(D,\un{u})}(X,\un{x})(Y,\un{y})(Z,\un{z})-
  \dint\Gamma_{(D,\un{u})}(Y,\un{y})(X,\un{x})(Z,\un{z})\\
  &+\Gamma_{(D,\un{u})}\bigl((X,\un{x}),\Gamma_{(D,\un{u})}(Y,\un{y})(Z,\un{z})\bigr)\hskip-2pt-
  \Gamma_{(D,\un{u})}\bigl((Y,\un{y}),\Gamma_{(D,\un{u})}(X,\un{x})(Z,\un{z})\bigr).
\end{align*}
We omit the details of the straightforward, but lengthy calculation of the curvature tensor.
\end{proof}

\begin{theorem}
For every point $(D,\un{u})\in\Xi_{n}$ and for every tangent vectors
  $(X,\un{x}),(Y,\un{y})\in\T_{n}$ the Ricci curvature tensor is
\begin{equation}
\label{eq:Ricci tensor}
\Ric_{(D,\un{u})}((X,\un{x}),(Y,\un{y}))=-\frac{n+1}{4}\Tr(D^{-1}XD^{-1}Y)
  +\frac{1}{4}\Tr(D^{-1}X)\Tr(D^{-1}Y)-\frac{\beta}{2(1+2n\alpha)}\langle\un{x},D\un{y}\rangle.
\end{equation}
\end{theorem}
\begin{proof}
At a point $(D,\un{u})\in\Xi_{n}$ for given tangent vectors $(X,\un{x}),(Y,\un{y})\in\T_{n}$
  the map
\begin{equation*}
R_{(D,\un{u})}(\cdot,(X,\un{x}),(Y,\un{y})):\T_{n}\to\T_{n}
  \quad (Z,\un{z})\mapsto R_{(D,\un{u})}((Z,\un{z}),(X,\un{x}),(Y,\un{y}))
\end{equation*}
is linear, and its trace is the Ricci tensor
\begin{equation*}
\Ric:\Xi_{n}\to\Lin(\T_{n}\times\T_{n},\mathbb{R})\quad
  (D,\un{u})\mapsto\Bigl(((X,\un{x}),(Y,\un{y}))\mapsto
  \Ric_{(D,\un{u})}((X,\un{x}),(Y,\un{y})) \Bigr),
\end{equation*}
where
\begin{equation*}
\Ric_{(D,\un{u})}((X,\un{x}),(Y,\un{y}))=\Tr R_{(D,\un{u})}(\cdot,(X,\un{x}),(Y,\un{y})).
\end{equation*}
The elements $\Ric_{(D,\un{u})}((X,\un{x}),(X,\un{x}))$ determines the Ricci tensor.
For the further calculation we fix the tangent vector $(X,\un{x})\in\T_{n}$.
According to the Equation (\ref{eq:Riemann curvature tensor}) the Riemann curvature tensor
  consists of six summands.
We compute the trace of the summands separately.
Let us denote by $E_{ij}$ the usual system of $n\times n$ matrix unit and define
\begin{equation*}
F_{ij}=E_{ij}+E_{ji}
\end{equation*}
  for indices $1\leq i<j\leq n$.
To compute the trace we choose the basis
\begin{equation}
\label{eq:basis}
\left\lbrace E_{ii}\right\rbrace_{i=1,\dots,n}\bigcup
   \left\lbrace F_{ij}\right\rbrace_{1\leq i<j\leq n}\bigcup
   \left\lbrace e_{i}\right\rbrace_{i=1,\dots,n}
\end{equation}
  in $\T_{n}$, where $(e_{i})_{i=1,\dots,n}$ is the canonical basis in $\mathbb{R}^{n}$.
The trace of the first summand is
\begin{align*}
\rho_{1}=&\frac{1}{2}\sum_{i=1}^{n}\Tr\left(XD^{-1}E_{ii}D^{-1}XE_{ii}
           -E_{ii}D^{-1}XD^{-1}XE_{ii}\right)\\
         &+\frac{1}{4}\sum_{1\leq i<j\leq n}
           \Tr\left(XD^{-1}F_{ij}D^{-1}XF_{ij}-F_{ij}D^{-1}XD^{-1}XF_{ij}\right)
\end{align*}
  which is
\begin{equation*}
\rho_{1}=-\frac{n+1}{4}\Tr D^{-1}XD^{-1}X+\frac{1}{4}\Tr X^{2}D^{-2}+\frac{1}{4}\Tr^{2}D^{-1}X.
\end{equation*}
The trace of the second summand is
\begin{equation*}
\rho_{2}=-\frac{\beta}{4}\sum_{i=1}^{n}\Tr(E_{ii}(E_{ii}\un{x}\odot D\un{x}))
  -\frac{\beta}{8}\sum_{1\leq i<j\leq n}\Tr F_{ij}(\un{x}\odot D\un{x})F_{ij}
  =-\beta\frac{n+1}{4}\langle\un{x},D\un{x}\rangle.
\end{equation*}
The third summand gives
\begin{equation*}
\rho_{3}=\frac{\alpha\beta}{1+2n\alpha}\sum_{i=1}^{n}
    \Tr\langle E_{ii}\un{x},\un{x}\rangle DE_{ii}
  +\frac{\alpha\beta}{1+2n\alpha}\sum_{1\leq i<j\leq n}
    \Tr\langle F_{ij}\un{x},\un{x}\rangle DF_{ij}
  =\frac{\alpha\beta}{1+2n\alpha}\langle\un{x},D\un{x}\rangle.
\end{equation*}
The trace of the forth, fifth and sixth summand is
\begin{align*}
\rho_{4}&=\frac{1}{4}\sum_{i=1}^{n}\langle e_{i},-D^{-1}XD^{-1}Xe_{i}\rangle
 =-\frac{1}{4}\Tr D^{-1}XD^{-1}X \\
\rho_{5}&=\frac{\beta}{4}\sum_{i=1}^{n}
  \langle e_{i},(\un{x}\odot e_{i})D\un{x}-(\un{x}\odot\un{x})De_{i}\rangle
  =\beta\frac{n-1}{4}\langle\un{x},D\un{x}\rangle\\
\rho_{6}&=\frac{\alpha\beta}{1+2n\alpha}\sum_{i=1}^{n}
  \langle e_{i},\langle\un{x},D\un{x}\rangle e_{i}-\langle e_{i},D\un{x}\rangle\un{x}\rangle
  =\frac{\alpha\beta(n-1)}{1+2n\alpha}\langle\un{x},D\un{x}\rangle.
\end{align*}
Adding the traces we have the diagonal element of the Ricci tensor
\begin{equation*}
\Ric_{(D,\un{u})}((X,\un{x}),(X,\un{x}))=\sum_{i=1}^{6}\rho_{i}
 =-\frac{n+1}{4}\Tr D^{-1}XD^{-1}X+\frac{1}{4}\Tr^{2}D^{-1}X
 -\frac{\beta}{1+2n\alpha}\langle\un{x},D\un{x}\rangle.
\end{equation*}
Using the polarization formula
\begin{equation*}
\Ric_{(D,\un{u})}((X,\un{x}),(Y,\un{y}))\hskip-2pt=\hskip-4pt\frac{1}{4}\Bigl(\hskip-1pt
  \Ric_{(D,\un{u})}((X+Y,\un{x}+\un{y}),(X+Y,\un{x}+\un{y}))\hskip-2pt-
  \Ric_{(D,\un{u})}((X-Y,\un{x}-\un{y}),(X-Y,\un{x}-\un{y}))\hskip-1pt\Bigr)
\end{equation*}
we get Equation (\ref{eq:Ricci tensor}).
\end{proof}

The next Theorem shows that the manifolds $\Xi_{n}$ and $\Xi_{n}^{(s)}$ has constant scalar
  curvature.

\begin{theorem}
For every point $D\in\Xi_{n}^{(s)}$ the scalar curvature of the space of special normal
  distributions is
\begin{equation}
\Scal_{s}(D)=-\frac{n\Bigl(2(n-1)(n+1)(n+2)\alpha+n^2+2n-1 \Bigr)}{4(1+2n\alpha)}
\end{equation}
  and for every point $(D,\un{u})\in\Xi_{n}$ the space of normal distributions is
\begin{equation}
\Scal(D,\un{u})=-\frac{n(n+1)\Bigl(2(n+2)(n-1)\alpha+n+1 \Bigr)}{4(1+2n\alpha)}.
\end{equation}
\end{theorem}
\begin{proof}
At a point $(D,\un{u})\in\Xi_{n}$ for given tangent vector $(X,\un{x})\in T_{n}$ the map
\begin{equation*}
\T_{n}\to\mathbb{R}\quad (Y,\un{y})\mapsto \Ric_{(D,\un{u})}((X,\un{x}),(Y,\un{y}))
\end{equation*}
  defines a linear functional.
So there exists a unique $(\Tilde{X},\tilde{\un{x}})\in\T_{n}$ tangent vector, such that
\begin{equation*}
g_{(D,\un{u})}((\Tilde{X},\tilde{\un{x}}),(Y,\un{y}))=\Ric_{(D,\un{u})}((X,\un{x}),(Y,\un{y}))
\end{equation*}
  holds for every tangent vector $(Y,\un{y})\in\T_{n}$.
Let us define the map
\begin{equation*}
\widetilde{\Ric}:\Xi_{n}\to\Lin(\T_{n},\T_{n})\quad
 (D,\un{u})\mapsto\Bigl((X,\un{x})\mapsto (\Tilde{X},\tilde{\un{x}})\Bigr).
\end{equation*}
The explicit expression
\begin{equation}
\label{eq:Ricci v tensor}
\widetilde{\Ric}_{(D,\un{u})}(X,\un{x})=-\frac{n+1}{2}X
  +\frac{1+2(n+1)\alpha}{2(1+2n\alpha)}D\Tr(D^{-1}X)-\frac{1}{2(1+2n\alpha)}\un{x}
\end{equation}
  can be easily verified.
The scalar curvature of the manifold is the trace of the map $\widetilde{\Ric}$
\begin{equation*}
\Scal:\Xi_{n}\to\mathbb{R}\quad (D,\un{u})\mapsto \Tr \widetilde{\Ric}_{D,\un{u}}.
\end{equation*}
Using the basis (\ref{eq:basis}) the trace of the three summand in
  the Equation (\ref{eq:Ricci v tensor}) is
\begin{align*}
\rho'_{1}&=-\frac{n+1}{2}\sum_{i=1}^{n}\Tr(E_{ii}E_{ii})
          -\frac{n+1}{4}\sum_{1\leq i<j\leq n}\Tr(F_{ij}F_{ij})=-\frac{n(n+1)^{2}}{4} \\
\rho'_{2}&=\frac{1+2(n+1)\alpha}{2(1+2n\alpha)}\sum_{i=1}^{n}\Tr(E_{ii}D)\Tr(E_{ii}D^{-1})
          +\frac{1+2(n+1)\alpha}{4(1+2n\alpha)}\sum_{1\leq i<j\leq n}
            \Tr(F_{ij}D)\Tr(F_{ij}D^{-1})\\
         &=n\frac{1+2(n+1)\alpha}{4(1+2n\alpha)}\\
\rho'_{3}&=-\frac{1}{2(1+2n\alpha)}\sum_{i=1}^{n}\langle e_{i},e_{i}\rangle
  =-\frac{n}{2(1+2n\alpha)}.
\end{align*}
The scalar curvature of the manifold $\Xi_{n}$ at a point $(D,\un{u})\in\Xi_{n}$ is
\begin{equation*}
\Scal(D,\un{u})=\rho'_{1}+\rho'_{2}+\rho'_{3}
\end{equation*}
  and the scalar curvature of the space of special normal distributions at a point
  $D\in\Xi_{n}^{(s)}$ is
\begin{equation*}
\Scal_{s}(D,\un{u})=\rho'_{1}+\rho'_{2}.
\end{equation*}
\end{proof}

\section{Conclusion}

Finally we have some remarks about the geometry of the generalized Gaussian distributions.

\begin{remark}
For every pair $(n,p)\in\mathcal{N}$, where $p<2$ the scalar curvature of the space of extended Gaussian
  distribution endowed with the Fisher information metric at every point is
\begin{equation*}
\Scal=-\frac{n(n+1)(2-p)}{4(2+n(p-1))} ((n+2)(n-1)(p-1)+2(n+1)).
\end{equation*}
We note that the parameter $p$ is in the interval $\left\rbrack \frac{n}{n+2},2\right\lbrack$,
  it means that the scalar curvature is a monotonously increasing function of $p$.
The scalar curvature at a given point is connected to the statistical distinguishability of the
  point from it's neighborhood, since the first nonconstant term in the Taylor expansion of the
  volume of the geodesic ball is the scalar curvature
\begin{equation*}
V_{n}(r)=\frac{r^{n}\pi^{n/2}}{\Gamma\left(\frac{n}{2}+1\right)}
  \left(1-\frac{\Scal}{6(n+2)}r^{2}+O(r^{4}) \right).
\end{equation*}
This idea is widely used in quantum information geometry and in that framework it is due to
  Petz \cite{Pet1}.
In this classical setting this means, that the parameter $p$ modulates the statistical properties
  of this manifold.
Namely, in the $p\to 2$ limit the manifold is more homogenous and it is more difficult to
  distinguish close points in the $p\to\frac{n}{n+2}$ limit; it is easier to decide whether two
  points are identical or just close to each other.
This can have relevance in hypothesis testing.
\end{remark}

\begin{remark}
Consider the space of special normal distributions and the Fisher information matrix
\begin{equation*}
g_{D}(X,Y)=\Tr\bigl(D^{-1}XD^{-1}Y\bigr).
\end{equation*}
Surprisingly from this well-known classical metric one can easily recover some metrics which are
  frequently used in quantum information theory.
In quantum setting just the trace one matrices of $\Xi_{n}^{(s)}$ are considered.
For example the Riemannian metrics
\begin{align*}
g^{(\KM)}_{D}(X,Y)&=\int_{0}^{\infty}g_{D+tE,\un{u}}(X,Y)\ \dint t\\
g^{(\LA)}_{D}(X,Y)&=g_{D}(D^{1/2}X,YD^{1/2})
\end{align*}
  are very important ones in quantum setting, they are called Kubo--Mori \cite{FicSau,Pet2}
  metric and largest metric.
This kind of differential geometrical connections can help to understand and to interpret
  the geometrical invariants of the quantum information manifolds.
\end{remark}

{\bf Acknowledgement.}
This work was supported by Japan Society for the Promotion of Science, contract number
  P 06917.

\end{document}